\documentclass[11pt]{amsart}
\usepackage{amsmath}
\usepackage[active]{srcltx}
\usepackage{t1enc}
\usepackage[latin2]{inputenc}
\usepackage{verbatim}
\usepackage{amsmath,amsfonts,amssymb,amsthm}
\usepackage[mathcal]{eucal}
\usepackage{enumerate}
\usepackage[centertags]{amsmath}
\usepackage{graphics}

\newtheorem{theorem}{Theorem}

\newtheorem*{G}{Theorem G}
\newtheorem*{M}{Theorem M}

\newtheorem*{GT1}{Theorem GG}
\newtheorem*{GT2}{Theorem GT}

\newtheorem*{kon}{Theorem K}

\begin{document}
\title[]{  logarithmic means of Walsh-Fourier Series}
\author{ Ushangi Goginava }
\address{U. Goginava, Department of Mathematics, Faculty of Exact and
Natural Sciences, Tbilisi State University, Chavcha\-vadze str. 1, Tbilisi
0179, Georgia}
\email{zazagoginava@gmail.com}
\thanks{.}

\begin{abstract}
In this paper we discuss some convergence and divergence properties of
subsequences of logarithmic means of Walsh-Fourier series . We give
necessary and sufficient conditions for the convergence regarding
logarithmic variation of numbers.
\end{abstract}

\maketitle

\bigskip \footnotetext{%
}

\section{Walsh Functions}

We shall denote the set of all non-negative integers by $\mathbb{N}$, the
set of all integers by$\,\,\mathbb{Z}$ and the set of dyadic rational
numbers in the unit interval $\mathbb{I}:=[0,1)$ by $\mathbb{Q}$. In
particular, each element of $\mathbb{Q}$ has the form $\frac{p}{2^{n}}$ for
some $p,n\in \mathbb{N},\,\,\,0\leq p\leq 2^{n}$.

Denote the dyadic expension of $n\in \mathbb{N}$ and $x\in \mathbb{I}$ by%
\begin{equation*}
n=\sum\limits_{j=0}^{\infty }\varepsilon _{j}\left( n\right)
2^{j},\varepsilon _{j}\left( n\right) =0,1
\end{equation*}%
and%
\begin{equation*}
x=\sum\limits_{j=0}^{\infty }\frac{x_{j}}{2^{j+1}},x_{j}=0,1.
\end{equation*}%
In the case of $x\in \mathbb{Q}$ chose the expension which terminates in
zeros. Define the dyadic addition $+$ as%
\begin{equation*}
x+y=\sum\limits_{k=0}^{\infty }\left\vert x_{k}-y_{k}\right\vert 2^{-(k+1)}.
\end{equation*}

The sets $I_{n}\left( x\right) :=\left\{ y\in \mathbb{I}%
:y_{0}=x_{0},...,y_{n-1}=x_{n-1}\right\} $ for $x\in \mathbb{I}%
,I_{n}:=I_{n}\left( 0\right) $ for $0<n\in \mathbb{N}$ and $I_{0}\left(
x\right) :=\mathbb{I}$ are the dyadic intervals of $\mathbb{I}$. For $0<n\in 
\mathbb{N}$ denote by $\left\vert n\right\vert :=\max \left\{ j\in \mathbb{N}%
:n_{j}\neq 0\right\} ,$ that is, $2^{\left\vert n\right\vert }\leq
n<2^{\left\vert n\right\vert +1}.$

The Rademacher system is defined by 
\begin{equation*}
\rho _{n}\left( x\right) :=\left( -1\right) ^{x_{n}}\text{ \ \ }\left( x\in 
\mathbb{I},n\in \mathbb{N}\right) .
\end{equation*}

The Walsh-Paley system is defined as the sequence of the Walsh-Paley
functions:%
\begin{equation*}
w_{n}\left( x\right) :=\prod\limits_{k=0}^{\infty }\left( \rho _{k}\left(
x\right) \right) ^{n_{k}}=\left( -1\right) ^{\sum\limits_{k=0}^{\left\vert
n\right\vert }n_{k}x_{k}},\left( x\in \mathbb{I},n\in \mathbb{N}\right) .
\end{equation*}

The Walsh-Dirichlet kernel is defined by 
\begin{equation*}
D_{n}\left( x\right) =\sum\limits_{k=0}^{n-1}w_{k}\left( x\right) .
\end{equation*}

Recall that (see \cite{sws}) 
\begin{equation}
D_{2^{n}}\left( x\right) =\left\{ 
\begin{array}{c}
2^{n},\mbox{if }x\in I_{n}\left( 0\right)  \\ 
0,\,\,\,\mbox{if }x\in \mathbb{I}\backslash I_{n}\left( 0\right) 
\end{array}%
\right. .  \label{dir2}
\end{equation}

As usual, denote by $L_{1}\left( \mathbb{I}\right) $ the set of measurable
functions defined on $\mathbb{I}$, for which%
\begin{equation*}
\left\Vert f\right\Vert _{1}:=\int\limits_{\mathbb{I}}\left\vert f\left(
t\right) \right\vert dt<\infty \text{.}
\end{equation*}%
Let\ $f\in L_{1}\left( \mathbb{I}\right) $. The partial sums of the
Walsh-Fourier series are defined as follows:

\begin{equation*}
S_{M}\left( x,f\right) :=\sum\limits_{i=0}^{M-1}\widehat{f}\left( i\right)
w_{i}\left( x\right) ,
\end{equation*}%
where the number 
\begin{equation*}
\widehat{f}\left( i\right) =\int\limits_{\mathbb{I}}f\left( t\right)
w_{i}\left( t\right) dt
\end{equation*}%
is said to be the $i$th Walsh-Fourier coefficient of the function\thinspace $%
f.$ Set $E_{n}\left( x,f\right) =S_{2^{n}}\left( x,f\right) .$The maximal
function is defined by 
\begin{equation*}
E^{\ast }\left( x,f\right) =\sup\limits_{n\in \mathbb{N}}E_{n}\left(
x,\left\vert f\right\vert \right) .
\end{equation*}

The notiation $a\lesssim b$ in the proofs stands for $a<c\cdot b$, where $c$
is an absolute constant.

\section{Logarithmic means}

In the literature, there is the notion of Riesz's logarithmic means of a
Fourier series. The $n$-th Riesz's logarithmic means of the Fourier series
of an integrable function $f$ is defined by 
\begin{equation*}
R_{n}\left( x,f\right) :=\frac{1}{l_{n}}\sum\limits_{k=1}^{n-1}\frac{%
S_{k}\left( x,f\right) }{k},
\end{equation*}%
where $l_{n}:=\sum_{k=1}^{n-1}\left( 1/k\right) $.

Riesz's logarithmic means with respect to the trigonometric system was
studied by a lot of authors. We mentioned, for instance, the paper by Szasz 
\cite{sz} and Yabuta \cite{ya}. This means with respect to the Walsh and
Vilenkin systems was discussed by Simon \cite{si}, Blahota, Gat \cite{BG},  G%
át \cite{gat}, Gát, Goginava \cite{GG}.

Let $\left\{ q_{k}:k\geq 0\right\} $ be a sequence of nonnegative numbers.
The $n$-th Nörlund means for the Fourier series of $f$ is defined by 
\begin{equation*}
\frac{1}{Q_{n}}\sum\limits_{k=0}^{n-1}q_{n-k}S_{k}f,
\end{equation*}%
where 
\begin{equation*}
Q_{n}:=\sum\limits_{k=1}^{n}q_{k}.
\end{equation*}%
If $q_{k}=k$, then we get the Nörlund logarithmic means 
\begin{equation*}
L_{n}\left( x,f\right) :=\frac{1}{l_{n}}\sum\limits_{k=1}^{n-1}\frac{%
S_{k}\left( x,f\right) }{n-k}.
\end{equation*}%
In this paper we call it logarithmic mean altough, it is a kind of
\textquotedblright reverse\textquotedblright\ Reisz's logarithmic mean.

It is easy to see that%
\begin{equation*}
L_{n}\left( x,f\right) =\int\limits_{\mathbb{I}}f\left( t\right) F_{n}\left(
x+t\right) dt,
\end{equation*}%
where by $F_{n}\left( t\right) $ we denote $n$th logarithmic kernel, i. e.%
\begin{equation*}
F_{n}\left( t\right) :=\frac{1}{l_{n}}\sum\limits_{k=1}^{n-1}\frac{%
D_{k}\left( t\right) }{n-k}.
\end{equation*}%
and Fej\'er kernel is defined by%
\begin{equation*}
K_{n}\left( t\right) :=\frac{1}{n}\sum\limits_{k=1}^{n}D_{k}\left( t\right) .
\end{equation*}

\section{$L_{1}$- estimation for logarithmic kernel}

For $n=\sum\limits_{j=0}^{\infty }\varepsilon _{j}\left( n\right)
2^{j},\varepsilon _{j}\left( n\right) =0,1$ we define 
\begin{equation*}
n\left( k\right) :=\sum\limits_{j=0}^{k}\varepsilon _{j}\left( n\right)
2^{j}.
\end{equation*}%
It is easy to see that $n\left( \left\vert n\right\vert \right) =n$. In this
paper for $L_{1}$-norm of logarithmic means we prove the following two sides
estimation.

\begin{theorem}
\label{est}Let $n\in \mathbb{N}$. Then%
\begin{equation*}
\left\Vert \frac{1}{l_{n}}\sum\limits_{j=1}^{n-1}\frac{D_{n-j}}{j}%
\right\Vert _{1}\sim \frac{1}{\left\vert n\right\vert }\sum\limits_{k=1}^{%
\left\vert n\right\vert }\left\vert \varepsilon _{k}\left( n\right)
-\varepsilon _{k+1}\left( n\right) \right\vert l_{n\left( k-1\right) }.
\end{equation*}
\end{theorem}

\begin{proof}
We can write%
\begin{eqnarray}
&&\sum\limits_{j=1}^{n-1}\frac{D_{n-j}\left( t\right) }{j}  \label{1} \\
&=&\sum\limits_{j=1}^{n\left( \left\vert n\right\vert -1\right) -1}\frac{%
D_{n-j}\left( t\right) }{j}+\sum\limits_{j=n\left( \left\vert n\right\vert
-1\right) }^{n-1}\frac{D_{n-j}\left( t\right) }{j}  \notag \\
&=&\sum\limits_{j=1}^{n\left( \left\vert n\right\vert -1\right) -1}\frac{%
D_{\varepsilon _{\left\vert n\right\vert }\left( n\right) 2^{\left\vert
n\right\vert }+n\left( \left\vert n\right\vert -1\right) -j}\left( t\right) 
}{j}+\sum\limits_{j=n\left( \left\vert n\right\vert -1\right) }^{n-1}\frac{%
D_{n-j}\left( t\right) }{j}.  \notag
\end{eqnarray}%
Since%
\begin{equation*}
D_{\varepsilon _{\left\vert n\right\vert }\left( n\right) 2^{\left\vert
n\right\vert }+n\left( \left\vert n\right\vert -1\right) -j}\left( t\right)
=\varepsilon _{\left\vert n\right\vert }\left( n\right) D_{2^{\left\vert
n\right\vert }}\left( t\right) +\left( w_{2^{\left\vert n\right\vert
}}\left( t\right) \right) ^{\varepsilon _{\left\vert n\right\vert }\left(
n\right) }D_{n\left( \left\vert n\right\vert -1\right) -j}\left( t\right) ,
\end{equation*}%
from (\ref{1}) we have%
\begin{eqnarray*}
\sum\limits_{j=1}^{n-1}\frac{D_{n-j}\left( t\right) }{j} &=&\varepsilon
_{\left\vert n\right\vert }\left( n\right) D_{2^{\left\vert n\right\vert
}}\left( t\right) l_{n\left( \left\vert n\right\vert -1\right) } \\
&&+\left( w_{2^{\left\vert n\right\vert }}\left( t\right) \right)
^{\varepsilon _{\left\vert n\right\vert }\left( n\right)
}\sum\limits_{j=1}^{n\left( \left\vert n\right\vert -1\right) -1}\frac{%
D_{n\left( \left\vert n\right\vert -1\right) -j}\left( t\right) }{j} \\
&&+\varepsilon _{\left\vert n\right\vert }\left( n\right)
\sum\limits_{j=0}^{2^{\left\vert n\right\vert }-1}\frac{D_{2^{\left\vert
n\right\vert }-j}\left( t\right) }{j+n\left( \left\vert n\right\vert
-1\right) }.
\end{eqnarray*}

Iterating this equality we obtain%
\begin{eqnarray}
&&\sum\limits_{j=1}^{n-1}\frac{D_{n-j}\left( t\right) }{j}  \label{2} \\
&=&\left( \sum\limits_{j=2}^{\left\vert n\right\vert }\varepsilon _{j}\left(
n\right) D_{2^{j}}\left( t\right) l_{n\left( j-1\right) }\right)
\prod\limits_{k=j+1}^{\left\vert n\right\vert }\left( \rho _{k}\left(
t\right) \right) ^{\varepsilon _{k}\left( n\right) }  \notag \\
&&+\left( \sum\limits_{j=2}^{\left\vert n\right\vert }\varepsilon _{j}\left(
n\right) \sum\limits_{k=0}^{2^{j}-1}\frac{D_{2^{j}-k}\left( t\right) }{%
k+n\left( j-1\right) }\right) \prod\limits_{s=j+1}^{\left\vert n\right\vert
}\left( \rho _{s}\left( t\right) \right) ^{\varepsilon _{s}\left( n\right) }
\notag \\
&&+\left( \sum\limits_{j=1}^{n\left( 1\right) -1}\frac{D_{n\left( 1\right)
-j}\left( t\right) }{j}\right) \prod\limits_{k=2}^{\left\vert n\right\vert
}\left( \rho _{k}\left( t\right) \right) ^{\varepsilon _{k}\left( n\right) }.
\notag
\end{eqnarray}%
Since%
\begin{equation*}
\varepsilon _{j}\left( n\right) D_{2^{j}}\left( t\right)
\prod\limits_{k=0}^{j}\left( \rho _{k}\left( t\right) \right) ^{\varepsilon
_{k}\left( n\right) }=\varepsilon _{j}\left( n\right) D_{2^{j}}\left(
t\right) \rho _{j}\left( t\right)
\end{equation*}%
we have%
\begin{eqnarray}
&&\left( \sum\limits_{j=2}^{\left\vert n\right\vert }\varepsilon _{j}\left(
n\right) D_{2^{j}}\left( t\right) l_{n\left( j-1\right) }\right)
\prod\limits_{k=j+1}^{\left\vert n\right\vert }\left( \rho _{k}\left(
t\right) \right) ^{\varepsilon _{k}\left( n\right) }  \label{3} \\
&=&w_{n}\left( t\right) \left( \sum\limits_{j=2}^{\left\vert n\right\vert
}\varepsilon _{j}\left( n\right) D_{2^{j}}\left( t\right) l_{n\left(
j-1\right) }\right) \prod\limits_{k=0}^{j}\left( \rho _{k}\left( t\right)
\right) ^{\varepsilon _{k}\left( n\right) }  \notag \\
&=&w_{n}\left( t\right) \left( \sum\limits_{j=2}^{\left\vert n\right\vert
}\varepsilon _{j}\left( n\right) D_{2^{j}}\left( t\right) \rho _{j}\left(
t\right) l_{n\left( j-1\right) }\right) .  \notag
\end{eqnarray}

Combining (\ref{2}) and (\ref{3}) we conclude that%
\begin{eqnarray}
&&\sum\limits_{j=1}^{n-1}\frac{D_{n-j}\left( t\right) }{j}  \label{H1-H3} \\
&=&w_{n}\left( t\right) \left( \sum\limits_{j=2}^{\left\vert n\right\vert
}\varepsilon _{j}\left( n\right) D_{2^{j}}\left( t\right) \rho _{j}\left(
t\right) l_{n\left( j-1\right) }\right)  \notag \\
&&+\left( \sum\limits_{j=2}^{\left\vert n\right\vert }\varepsilon _{j}\left(
n\right) \sum\limits_{k=0}^{2^{j}-1}\frac{D_{2^{j}-k}\left( t\right) }{%
k+n\left( j-1\right) }\right) \prod\limits_{s=j+1}^{\left\vert n\right\vert
}\left( \rho _{s}\left( t\right) \right) ^{\varepsilon _{s}\left( n\right) }
\notag \\
&&+\left( \sum\limits_{j=1}^{n\left( 1\right) -1}\frac{D_{n\left( 1\right)
-j}\left( t\right) }{j}\right) \prod\limits_{k=2}^{\left\vert n\right\vert
}\left( \rho _{k}\left( t\right) \right) ^{\varepsilon _{k}\left( n\right) }
\notag \\
&=&:H_{n}^{\left( 1\right) }\left( t\right) +H_{n}^{\left( 2\right) }\left(
t\right) +H_{n}^{\left( 3\right) }\left( t\right) .  \notag
\end{eqnarray}

Since%
\begin{equation*}
D_{2^{j}-k}\left( t\right) =D_{2^{j}}\left( t\right) -w_{2^{j}-1}\left(
t\right) D_{k}\left( t\right) ,k=1,2,...,2^{j}-1
\end{equation*}%
for $H_{n}^{\left( 2\right) }\left( t\right) $ we can write%
\begin{eqnarray}
&&H_{n}^{\left( 2\right) }\left( t\right)  \label{H21-H23} \\
&=&\left( \sum\limits_{j=2}^{\left\vert n\right\vert }\varepsilon _{j}\left(
n\right) \frac{D_{2^{j}}\left( t\right) }{n\left( j-1\right) }\right)
\prod\limits_{s=j+1}^{\left\vert n\right\vert }\left( \rho _{s}\left(
t\right) \right) ^{\varepsilon _{s}\left( n\right) }  \notag \\
&&+\left( \sum\limits_{j=2}^{\left\vert n\right\vert }\varepsilon _{j}\left(
n\right) D_{2^{j}}\left( t\right) \sum\limits_{k=1}^{2^{j}-1}\frac{1}{%
k+n\left( j-1\right) }\right) \prod\limits_{s=j+1}^{\left\vert n\right\vert
}\left( \rho _{s}\left( t\right) \right) ^{\varepsilon _{s}\left( n\right) }
\notag \\
&&-\left( \sum\limits_{j=2}^{\left\vert n\right\vert }\varepsilon _{j}\left(
n\right) w_{2^{j}-1}\left( t\right) \sum\limits_{k=1}^{2^{j}-1}\frac{%
D_{k}\left( t\right) }{k+n\left( j-1\right) }\right)
\prod\limits_{s=j+1}^{\left\vert n\right\vert }\left( \rho _{s}\left(
t\right) \right) ^{\varepsilon _{s}\left( n\right) }  \notag \\
&=&:H_{n}^{\left( 21\right) }\left( t\right) +H_{n}^{\left( 22\right)
}\left( t\right) +H_{n}^{\left( 23\right) }\left( t\right) .  \notag
\end{eqnarray}

From (\ref{dir2}) we have%
\begin{equation}
\left\Vert H_{n}^{\left( 21\right) }\right\Vert _{1}\leq
\sum\limits_{j=2}^{\left\vert n\right\vert }\left\Vert D_{2^{j}}\right\Vert
_{1}<\left\vert n\right\vert .  \label{H21}
\end{equation}

Since%
\begin{eqnarray*}
H_{n}^{\left( 22\right) }\left( t\right) &=&\left(
\sum\limits_{j=2}^{\left\vert n\right\vert }\varepsilon _{j}\left( n\right)
D_{2^{j}}\left( t\right) \left( l_{n\left( j\right) }-l_{n\left( j-1\right)
}-\frac{1}{n\left( j-1\right) }\right) \right) \\
&&\times \prod\limits_{s=j+1}^{\left\vert n\right\vert }\left( \rho
_{s}\left( t\right) \right) ^{\varepsilon _{s}\left( n\right) }
\end{eqnarray*}%
from (\ref{dir2}) we get%
\begin{equation}
\left\Vert H_{n}^{\left( 22\right) }\right\Vert _{1}\leq
\sum\limits_{j=2}^{\left\vert n\right\vert }\varepsilon _{j}\left( n\right)
\left( l_{n\left( j\right) }-l_{n\left( j-1\right) }\right) +c\left\vert
n\right\vert \leq c\left\vert n\right\vert .  \label{H22}
\end{equation}

Usin Abel's transformation we obtain%
\begin{eqnarray*}
&&\sum\limits_{k=1}^{2^{j}-1}\frac{D_{k}\left( t\right) }{k+n\left(
j-1\right) } \\
&=&\sum\limits_{k=1}^{2^{j}-2}\left( \frac{1}{k+n\left( j-1\right) }-\frac{1%
}{k+1+n\left( j-1\right) }\right) kK_{k}\left( t\right) \\
&&+\frac{2^{j}-1}{2^{j}-1+n\left( j-1\right) }K_{2^{j}-1}\left( t\right) .
\end{eqnarray*}%
Since (see \cite{sws}) $\sup\limits_{n}\left\Vert K_{n}\right\Vert <\infty $
for $H_{n}^{\left( 23\right) }\left( t\right) $ we can write%
\begin{eqnarray}
&&\left\Vert H_{n}^{\left( 23\right) }\right\Vert _{1}  \label{H23} \\
&\lesssim &\sum\limits_{j=2}^{\left\vert n\right\vert }\varepsilon
_{j}\left( n\right) \sum\limits_{k=0}^{2^{j}-1}\left( \frac{1}{k+n\left(
j-1\right) }-\frac{1}{k+1+n\left( j-1\right) }\right) k  \notag \\
&&+\sum\limits_{j=2}^{\left\vert n\right\vert }\varepsilon _{j}\left(
n\right) \frac{2^{j}-1}{2^{j}-1+n\left( j-1\right) }  \notag \\
&\lesssim &\sum\limits_{j=2}^{\left\vert n\right\vert }\varepsilon
_{j}\left( n\right) \sum\limits_{k=1}^{2^{j}-1}\frac{1}{\left( k+n\left(
j-1\right) \right) }  \notag \\
&&+\sum\limits_{j=2}^{\left\vert n\right\vert }\varepsilon _{j}\left(
n\right) \sum\limits_{k=1}^{2^{j}-1}\frac{n\left( j-1\right) }{\left(
k+n\left( j-1\right) \right) ^{2}}+\sum\limits_{j=2}^{\left\vert
n\right\vert }\varepsilon _{j}\left( n\right) \frac{2^{j}-1}{2^{j}-1+n\left(
j-1\right) }  \notag \\
&\lesssim &\sum\limits_{j=2}^{\left\vert n\right\vert }\varepsilon
_{j}\left( n\right) \left( l_{n\left( j\right) }-l_{n\left( j-1\right)
}\right) +\left\vert n\right\vert \lesssim \left\vert n\right\vert .  \notag
\end{eqnarray}

Combining (\ref{H21-H23})-(\ref{H23}) we conclude that%
\begin{equation}
\left\Vert H_{n}^{\left( 2\right) }\right\Vert _{1}\leq c\left\vert
n\right\vert .  \label{H2}
\end{equation}%
It is easy to see that%
\begin{equation}
\sup\limits_{n}\left\Vert H_{n}^{\left( 3\right) }\right\Vert _{1}\leq c.
\label{H3}
\end{equation}

First, we find upper estimation for $\left\Vert H_{n}^{\left( 1\right)
}\right\Vert _{1}$. We can write%
\begin{eqnarray*}
H_{n}^{\left( 1\right) }\left( t\right) &=&w_{n}\left( t\right) \left(
\sum\limits_{j=2}^{\left\vert n\right\vert }\varepsilon _{j}\left( n\right)
l_{n\left( j-1\right) }\left( D_{2^{j+1}}\left( t\right) -D_{2^{j}}\left(
t\right) \right) \right) \\
&=&w_{n}\left( t\right) \left( \sum\limits_{j=2}^{\left\vert n\right\vert
-1}\left( \varepsilon _{j}\left( n\right) l_{n\left( j-1\right)
}-\varepsilon _{j+1}\left( n\right) l_{n\left( j\right) }\right)
D_{2^{j+1}}\left( t\right) \right) \\
&&+w_{n}\left( t\right) l_{n\left( \left\vert n\right\vert -1\right)
}D_{2^{\left\vert n\right\vert +1}}-w_{n}\left( t\right) \varepsilon
_{2}\left( n\right) l_{n\left( 1\right) }D_{2^{2}}\left( t\right) .
\end{eqnarray*}%
Hence, from (\ref{dir2}) we obtain%
\begin{eqnarray}
\left\Vert H_{n}^{\left( 1\right) }\right\Vert _{1} &\leq
&\sum\limits_{j=2}^{\left\vert n\right\vert -1}\left\vert \varepsilon
_{j}\left( n\right) l_{n\left( j-1\right) }-\varepsilon _{j+1}\left(
n\right) l_{n\left( j\right) }\right\vert +c\left\vert n\right\vert
\label{H1-1} \\
&\leq &\sum\limits_{j=2}^{\left\vert n\right\vert -1}\left\vert \varepsilon
_{j}\left( n\right) -\varepsilon _{j+1}\left( n\right) \right\vert
l_{n\left( j-1\right) }  \notag \\
&&+\sum\limits_{j=2}^{\left\vert n\right\vert -1}\varepsilon _{j+1}\left(
n\right) \left( l_{n\left( j\right) }-l_{n\left( j-1\right) }\right)
+c\left\vert n\right\vert  \notag \\
&\leq &\sum\limits_{j=2}^{\left\vert n\right\vert }\left\vert \varepsilon
_{j}\left( n\right) -\varepsilon _{j+1}\left( n\right) \right\vert
l_{n\left( j-1\right) }+c\left\vert n\right\vert .  \notag
\end{eqnarray}

Now, we find lower estimation for $\left\Vert H_{n}^{\left( 1\right)
}\right\Vert _{1}$. Let $a_{i}$ and $b_{i},i=1,...,s$ be strictly increasing
sequences, i. e.%
\begin{equation*}
0\leq a_{1}\leq b_{1}<a_{2}\leq b_{2}<\cdots <a_{s}\leq b_{s}<a_{s+1}=\infty
\end{equation*}%
for which%
\begin{equation*}
\varepsilon _{j}\left( n\right) =\left\{ 
\begin{array}{l}
1,a_{i}\leq j\leq b_{i} \\ 
0,b_{i}<j<a_{i+1}%
\end{array}%
\right. .
\end{equation*}%
Then it is evident that%
\begin{equation}
b_{j}+1<a_{j+1}.  \label{b<a}
\end{equation}

Set%
\begin{equation*}
A_{k}:=\left( \frac{1}{2^{a_{k}+1}},\frac{1}{2^{a_{k}}}\right)
,B_{k}:=\left( \frac{1}{2^{b_{k}+2}},\frac{1}{2^{b_{k}+1}}\right) ,k=1,...,s.
\end{equation*}%
Let $x\in A_{k}$. Then we can write%
\begin{eqnarray*}
\left\vert H_{n}^{\left( 1\right) }\left( t\right) \right\vert 
&=&\left\vert \sum\limits_{j=2}^{\left\vert n\right\vert }\varepsilon
_{j}\left( n\right) \left( D_{2^{j+1}}\left( t\right) -D_{2^{j}}\left(
t\right) \right) l_{n\left( j-1\right) }\right\vert  \\
&=&\left\vert \sum\limits_{i=1}^{k-1}\sum\limits_{j=a_{i}}^{b_{i}}\left(
D_{2^{j+1}}\left( t\right) -D_{2^{j}}\left( t\right) \right) l_{n\left(
j-1\right) }\right.  \\
&&+\left. \sum\limits_{j=a_{k}}^{b_{k}}\left( D_{2^{j+1}}\left( t\right)
-D_{2^{j}}\left( t\right) \right) l_{n\left( j-1\right) }\right\vert  \\
&=&\left\vert
\sum\limits_{i=1}^{k-1}\sum\limits_{j=a_{i}}^{b_{i}}2^{j}l_{n\left(
j-1\right) }-2^{a_{k}}l_{n\left( a_{k}-1\right) }\right\vert .
\end{eqnarray*}

From (\ref{b<a}) we can write%
\begin{eqnarray*}
\sum\limits_{i=1}^{k-1}\sum\limits_{j=a_{i}}^{b_{i}}2^{j}l_{n\left(
j-1\right) } &\leq &l_{n\left( b_{k-1}-1\right)
}\sum\limits_{i=1}^{k-1}\left( 2^{b_{i}+1}-2^{a_{i}}\right) \\
&\leq &l_{n\left( b_{k-1}-1\right) }\sum\limits_{i=1}^{k-1}\left(
2^{b_{i}+1}-2^{b_{i-1}+1}\right) \\
&\leq &2^{b_{k-1}+1}l_{n\left( b_{k-1}-1\right) } \\
&\leq &2^{b_{k-1}+1}l_{n\left( a_{k}-1\right) }.
\end{eqnarray*}%
Consequently,%
\begin{equation*}
\left\vert H_{n}^{\left( 1\right) }\left( t\right) \right\vert \geq
2^{a_{k}}l_{n\left( a_{k}-1\right) }-2^{b_{k-1}+1}l_{n\left( a_{k}-1\right)
}\geq 2^{a_{k}-1}l_{n\left( a_{k}-1\right) }.
\end{equation*}%
Integrating on $A_{k}$ we get%
\begin{equation}
\int\limits_{A_{k}}\left\vert H_{n}^{\left( 1\right) }\left( t\right)
\right\vert dt\geq \int\limits_{A_{k}}2^{a_{k}-1}l_{n\left( a_{k}-1\right)
}dt=\frac{l_{n\left( a_{k}-1\right) }}{4}.  \label{A}
\end{equation}

On the interval $B_{k}$ we have%
\begin{eqnarray*}
\left\vert H_{n}^{\left( 1\right) }\left( t\right) \right\vert &=&\left\vert
\sum\limits_{i=1}^{k}\sum\limits_{j=a_{i}}^{b_{i}}\left( D_{2^{j+1}}\left(
t\right) -D_{2^{j}}\left( t\right) \right) l_{n\left( j-1\right) }\right\vert
\\
&=&\sum\limits_{i=1}^{k}\sum\limits_{j=a_{i}}^{b_{i}}2^{j}l_{n\left(
j-1\right) }\geq l_{\left( b_{k}-1\right) }2^{b_{k}}.
\end{eqnarray*}%
Hence,%
\begin{equation}
\int\limits_{B_{k}}\left\vert H_{n}^{\left( 1\right) }\left( t\right)
\right\vert dt\geq \int\limits_{B_{k}}l_{\left( b_{k}-1\right) }2^{b_{k}}dt=%
\frac{l_{n\left( b_{k}-1\right) }}{4}.  \label{B}
\end{equation}

Since $A_{i},B_{i},i=1,...,s$ are pairwise disjoint from (\ref{A}) and (\ref%
{B}) we have%
\begin{eqnarray}
\int\limits_{\mathbb{I}}\left\vert H_{n}^{\left( 1\right) }\left( t\right)
\right\vert dt &\geq &\sum\limits_{k=1}^{s}\left(
\int\limits_{A_{k}}\left\vert H_{n}^{\left( 1\right) }\left( t\right)
\right\vert dt+\int\limits_{B_{k}}\left\vert H_{n}^{\left( 1\right) }\left(
t\right) \right\vert dt\right)  \label{H1} \\
&\geq &\frac{1}{4}\sum\limits_{k=1}^{s}\left( l_{n\left( a_{k}-1\right)
}+l_{n\left( b_{k}-1\right) }\right)  \notag \\
&=&\frac{1}{4}\sum\limits_{k=1}^{\left\vert n\right\vert }\left\vert
\varepsilon _{k}\left( n\right) -\varepsilon _{k+1}\left( n\right)
\right\vert l_{n\left( k-1\right) }.  \notag
\end{eqnarray}

Combining (\ref{H1-H3})-(\ref{H1-1}) and (\ref{H1}) we complete the proof of
Theorem \ref{est}.
\end{proof}

\section{Almost Everywhere Convergence of logarithmic Means}

For a non-negative integer $n$ let us denote%
\begin{equation*}
V_{S}\left( n\right) :=\sum\limits_{i=0}^{\infty }\left\vert \varepsilon
_{i}\left( n\right) -\varepsilon _{i+1}\left( n\right) \right\vert
+\varepsilon _{0}\left( n\right)
\end{equation*}%
and%
\begin{equation*}
V_{L}\left( n\right) :=\frac{1}{\left\vert n\right\vert }\sum\limits_{k=1}^{%
\left\vert n\right\vert }\left\vert \varepsilon _{k}\left( n\right)
-\varepsilon _{k+1}\left( n\right) \right\vert l_{n\left( k-1\right) }.
\end{equation*}

It is known that if $n_{j}<n_{j+1}$, 
\begin{equation}
\sup\limits_{j}V_{s}\left( n_{j}\right) <\infty ,  \label{Vs}
\end{equation}%
then a. e. $S_{n_{j}}\left( f\right) \rightarrow f$. On the other hand,
Koniagin \cite{kon} proved that the condition (\ref{Vs}) is not necessary
for a. e. convergence of subsequence of partial sums. Moreover, he gave
negative answer to the question of Balashov and proved the validity of the
following theorem.

\begin{kon}[Konyagin]
Suppose $\left\{ n_{A}\right\} $ is an increasing sequence of natural
numbers, $k_{A}:=\left[ \log _{2}n_{A}\right] +1,$ and $2^{k_{A}}$ is a
divider of $n_{A+1}$ for all $A$. Then $S_{n_{A}}\left( f\right) \rightarrow
f$ a. e. for any function $f\in L_{1}\left( \mathbb{I}\right) $.
\end{kon}

For instance, a sequence $\left\{ n_{A}\right\}
,n_{A}:=2^{A^{2}}\sum\limits_{i=0}^{A}4^{i},$ such that $\sup%
\limits_{n_{A}}V\left( n_{A}\right) =\infty ,$ satisfies the hypotheses of
the theorem.

Almost ewerywhere convergence of $\left\{ t_{2^{A}}\left( f\right) :A\geq
1\right\} $ with respect to Walsh-Paley system was studied by author \cite%
{gogamapn}. In particular, the following is proved

\begin{G}
Let $f\in L_{1}\left( \mathbb{I}\right) $. Then $t_{2^{A}}\left( x,f\right)
\rightarrow f\left( x\right) $ as $A\rightarrow \infty $ a. e. $x\in \mathbb{%
I}$.
\end{G}

Nagy in \cite{nagy} established a similar result for the Walsh-Kaczmarz
system. However, a divergence on the set with positive measure for the whole
sequence $\left\{ t_{n}\left( f\right) :n\geq 1\right\} $ was proved by
G\'at and Goginava \cite{gat-gogi}. Memic \cite{memic} improved Theorem G
and proved that the following is true.

\begin{M}
Let $f\in L_{1}\left( \mathbb{I}\right) $ and%
\begin{equation}
\sup\limits_{A}\frac{1}{\left\vert m_{A}\right\vert }\sum\limits_{k=1}^{%
\left\vert m_{A}\right\vert }\varepsilon _{k}\left( m_{A}\right)
l_{m_{A}\left( k-1\right) }<\infty .  \label{mem}
\end{equation}%
Then $t_{m_{A}}\left( x,f\right) \rightarrow f\left( x\right) $ as $%
A\rightarrow \infty $ for a. e. $x\in \mathbb{I}$.
\end{M}

In this paper we are going to replace condition (\ref{mem}) with more weaker
condition%
\begin{equation}
\sup\limits_{A}\frac{1}{\left\vert m_{A}\right\vert }\sum\limits_{k=1}^{%
\left\vert m_{A}\right\vert }\left\vert \varepsilon _{k}\left( m_{A}\right)
-\varepsilon _{k+1}\left( m_{A}\right) \right\vert l_{m_{A}\left( k-1\right)
}<\infty .  \label{new}
\end{equation}

It is easy to see that condition (\ref{mem}) imply condition (\ref{new}), on
other hand, for the sequence $\left\{ 2^{A}-1:A\in \mathbb{N}\right\} $
condition (\ref{mem}) does not holds and condition (\ref{new}) holds. So, we
prove that the following is valid.

\begin{theorem}
\label{a.e.}Let $f\in L_{1}\left( \mathbb{I}\right) $ and condition (\ref%
{new}) is holds. Then $t_{m_{A}}\left( x,f\right) \rightarrow f\left(
x\right) $ as $A\rightarrow \infty $ for a. e. $x\in \mathbb{I}$.
\end{theorem}

\begin{proof}
From (\ref{H1-H3}) we have%
\begin{equation}
f\ast \left( l_{m_{A}}F_{m_{A}}\right) \left( x\right) =f\ast
H_{m_{A}}^{\left( 1\right) }\left( x\right) +f\ast H_{m_{A}}^{\left(
2\right) }\left( x\right) +f\ast H_{m_{A}}^{\left( 3\right) }\left( x\right)
.  \label{f1-f3}
\end{equation}%
It is easy to see that%
\begin{equation}
\left\Vert \sup\limits_{A}\left\vert f\ast H_{m_{A}}^{\left( 3\right)
}\right\vert \right\Vert _{1}\lesssim \left\Vert f\right\Vert _{1}.
\label{f3}
\end{equation}

From (\ref{H21-H23}) we can write%
\begin{equation}
f\ast H_{m_{A}}^{\left( 2\right) }\left( x\right) =f\ast H_{m_{A}}^{\left(
21\right) }\left( x\right) +f\ast H_{m_{A}}^{\left( 22\right) }\left(
x\right) +f\ast H_{m_{A}}^{\left( 22\right) }\left( x\right) .
\label{f21-f23}
\end{equation}%
Using (\ref{dir2}) we have%
\begin{equation*}
\left\vert f\ast H_{m_{A}}^{\left( 21\right) }\left( x\right) \right\vert
\leq \sum\limits_{j=1}^{\left\vert m_{A}\right\vert -1}\left( \left\vert
f\right\vert \ast D_{2^{j}}\left( x\right) \right) \leq \left\vert
m_{A}\right\vert E^{\ast }\left( x,f\right) .
\end{equation*}%
Hence%
\begin{equation}
\sup\limits_{A}\frac{\left\vert f\ast H_{m_{A}}^{\left( 21\right) }\left(
x\right) \right\vert }{\left\vert m_{A}\right\vert }\leq E^{\ast }\left(
x,f\right) .  \label{sup21}
\end{equation}

We can write%
\begin{eqnarray*}
\left\vert f\ast H_{m_{A}}^{\left( 22\right) }\left( x\right) \right\vert 
&\leq &\sum\limits_{j=1}^{\left\vert m_{A}\right\vert -1}\varepsilon
_{j}\left( m_{A}\right) \left( l_{m_{A}\left( j\right) }-l_{m_{A}\left(
j-1\right) }\right) \left( \left\vert f\right\vert \ast D_{2^{j}}\left(
x\right) \right)  \\
&\lesssim &\left\vert m_{A}\right\vert E^{\ast }\left( x,f\right) ,
\end{eqnarray*}%
\begin{equation}
\sup\limits_{A}\frac{\left\vert f\ast H_{m_{A}}^{\left( 22\right) }\left(
x\right) \right\vert }{\left\vert m_{A}\right\vert }\lesssim E^{\ast }\left(
x,f\right) .  \label{sup22}
\end{equation}

It is proved in \cite{gat-gogi}%
\begin{equation}
\sup\limits_{\lambda }\lambda \left\vert \left\{ \sup\limits_{A}\frac{%
\left\vert f\ast H_{m_{A}}^{\left( 23\right) }\right\vert }{\left\vert
m_{A}\right\vert }>\lambda \right\} \right\vert \lesssim \left\Vert
f\right\Vert _{1}.  \label{sup23}
\end{equation}

Since $\sup\limits_{\lambda }\lambda \left\vert \left\{ E^{\ast }\left(
f\right) >\lambda \right\} \right\vert \lesssim \left\Vert f\right\Vert _{1}$
from (\ref{sup21})- (\ref{sup23}) we get%
\begin{equation}
\sup\limits_{\lambda }\lambda \left\vert \left\{ \sup\limits_{A}\frac{%
\left\vert f\ast H_{m_{A}}^{\left( 2\right) }\right\vert }{\left\vert
m_{A}\right\vert }>\lambda \right\} \right\vert \lesssim \left\Vert
f\right\Vert _{1}.  \label{sup2}
\end{equation}

Now, we estimate $\left\vert f\ast H_{m_{A}}^{\left( 1\right) }\left(
x\right) \right\vert $. We have%
\begin{eqnarray*}
\left\vert f\ast H_{m_{A}}^{\left( 1\right) }\left( x\right) \right\vert 
&\lesssim &\sum\limits_{j=1}^{\left\vert m_{A}\right\vert -1}\left(
\varepsilon _{j}\left( m_{A}\right) l_{m_{A}\left( j-1\right) }-\varepsilon
_{j+1}\left( m_{A}\right) l_{m_{A}\left( j\right) }\right) \left( \left\vert
f\right\vert \ast D_{2^{j}}\left( x\right) \right)  \\
&&+l_{\left( \left\vert m_{A}\right\vert -1\right) }\left( \left\vert
f\right\vert \ast D_{2^{\left\vert m_{A}\right\vert +1}}\left( x\right)
\right)  \\
&\lesssim &E^{\ast }\left( x,f\right) \left( l_{\left( \left\vert
m_{A}\right\vert -1\right) }+\sum\limits_{j=1}^{\left\vert m_{A}\right\vert
-1}\left( \varepsilon _{j}\left( m_{A}\right) l_{m_{A}\left( j-1\right)
}-\varepsilon _{j+1}\left( m_{A}\right) l_{m_{A}\left( j\right) }\right)
\right)  \\
&\lesssim &E^{\ast }\left( x,f\right) \left( l_{\left( \left\vert
m_{A}\right\vert -1\right) }+\sum\limits_{j=1}^{\left\vert m_{A}\right\vert
-1}\left\vert \varepsilon _{j}\left( m_{A}\right) -\varepsilon _{j+1}\left(
m_{A}\right) \right\vert l_{m_{A}\left( j-1\right) }\right.  \\
&&\left. +\sum\limits_{j=1}^{\left\vert m_{A}\right\vert -1}\varepsilon
_{j+1}\left( m_{A}\right) \left( l_{m_{A}\left( j\right) }-l_{m_{A}\left(
j-1\right) }\right) \right)  \\
&\lesssim &E^{\ast }\left( x,f\right) \left( l_{\left( \left\vert
m_{A}\right\vert -1\right) }+\sum\limits_{j=1}^{\left\vert m_{A}\right\vert
-1}\left\vert \varepsilon _{j}\left( m_{A}\right) -\varepsilon _{j+1}\left(
m_{A}\right) \right\vert l_{m_{A}\left( j-1\right) }\right) .
\end{eqnarray*}%
From the condition of the Theorem we can write%
\begin{equation*}
\sup\limits_{A}\frac{\left\vert f\ast H_{m_{A}}^{\left( 1\right) }\left(
x\right) \right\vert }{l_{m_{A}}}\lesssim E^{\ast }\left( x,f\right)
V_{L}\left( m_{A}\right) 
\end{equation*}%
and%
\begin{equation}
\sup\limits_{\lambda }\lambda \left\vert \left\{ \sup\limits_{A}\frac{%
\left\vert f\ast H_{m_{A}}^{\left( 1\right) }\right\vert }{\left\vert
m_{A}\right\vert }>\lambda \right\} \right\vert \lesssim \left\Vert
f\right\Vert _{1}.  \label{sup1}
\end{equation}

Combining (\ref{H21-H23}), (\ref{f3}), (\ref{sup2}) and (\ref{sup1}) we
conclude that 
\begin{equation*}
\sup\limits_{\lambda }\lambda \left\vert \left\{ \sup\limits_{A}\frac{%
\left\vert f\ast F_{m_{A}}\right\vert }{\left\vert m_{A}\right\vert }%
>\lambda \right\} \right\vert \lesssim \left\Vert f\right\Vert _{1}.
\end{equation*}

By the well-known density argument we complete the proof of Theorem \ref%
{a.e.}.
\end{proof}

\section{Uniform and $L$-Convergence of Logarithmic Means}

Denote by $C_{w}\left( \mathbb{I}\right) $ the space of uniformly continuous
functions on $\mathbb{I}$, with the supremum norm 
\begin{equation*}
\left\Vert f\right\Vert _{C_{w}}:=\sup\limits_{x\in \mathbb{I}}\left\vert
f\left( x\right) \right\vert \text{ \ \ \ \ }\left( f\in C_{w}\left( \mathbb{%
I}\right) \right) .
\end{equation*}%
Let $X=X\left( \mathbb{I}\right) $ be either the space $L_{1}\left( \mathbb{I%
}\right) $, or the space of uniformly continuous functions, that is, $%
C_{w}\left( \mathbb{I}\right) $. The corresponding norm is denoted by $%
\left\Vert \cdot \right\Vert _{X}$. The modulus of continuity, when $X=C$,
and the integrated modulus of continuity, when $X=L_{1}$ is defined by%
\begin{equation*}
\omega \left( \delta ,f\right) :=\sup\limits_{0<h\leq \delta }\left\Vert
f\left( \cdot +h\right) -f\left( \cdot \right) \right\Vert _{X}.
\end{equation*}

If $\omega \left( \delta \right) $ is a modulus of continuity, then $%
H_{X}^{\omega }$ denotes the class of functions $f\in X\left( \mathbb{I}%
\right) $ for which $\omega \left( \delta ,f\right) =O\left( \omega \left(
\delta \right) \right) $ as $\delta \rightarrow 0+$.

For Walsh-Fourier series Fine \cite{fine} has obtained a sufficient
condition for the uniform convergence which is in a complete analogy with
the Dini-Lipshitz condition (see also \cite{sws}). Similar results are true
for the space of integrable functions $L_{1}\left( \mathbb{I}\right) $ \cite%
{onn}. Gulicev \cite{gul} has estimated the rate of uniform convergence of a
Walsh-Fourier series using Lebesgue constant and modulus of continuity.
Uniform convergence of Walsh-Fourier series of the functions of classes of
generalized bounded variation was investigated by author \cite{gogAMH}. This
problem has been considered for Vilenkin group by Fridli \cite{fr} and G\'at 
\cite{gat2}. Lukomskii \cite{luk1} considered uniform and $L_{1}$%
-convergence of subsequence of partial sums of Walsh-Fourier series. In
particular, he proved that the condition $\sup\limits_{A}V_{S}\left(
m_{A}\right) <\infty $ is necessary and sufficient condition for uniform and 
$L_{1}$-convergence of subsequence of partial sums $S_{m_{A}}\left( f\right) 
$ of Walsh-Fourier series. In Moricz and Siddiqi \cite{mor} investigated
approximation properties of Nörlund means $\frac{1}{Q_{n}}%
\sum\limits_{k=0}^{n-1}q_{n-k}S_{k}f$. The case when we have $q_{k}:=1/k$
differs from the types discussed by Moricz and Siddiqi in \cite{mor}. His
method is not applicable for logarithmic means. In \cite{gat-gogi2} it is
proved that Theorem of Moricz does not hold for $L_{1},$ $C_{w}$ and $%
q_{k}:=1/k.$ In particular, the following is proved.

\begin{GT1}
The following conditions are equivalent:\newline
\newline
a) \ \ $\omega \left( \delta \right) =o\left( \frac{1}{\log \left( 1/\delta
\right) }\right) ;$\newline
\newline
b) Nörlund logarithmic means of Walsh-Fourier series for all functions from
the $H_{X}^{\omega }$ converge in $X$-norm;\newline
\newline
c) partial sums of Walsh-Fourier series for all functions from the $%
H_{X}^{\omega }$ converge in $X$-norm;.
\end{GT1}

In \cite{gog-tk} it is investigated $X$-norm convergence of subsequence of
logarithmic means of Walsh-Fourier series. In particular the following are
proved.

\begin{GT2}
a) Let $f\in X\left( \mathbb{I}\right) $ and%
\begin{equation}
\sup\limits_{A}\frac{\log \left( m_{A}-2^{\left\vert m_{A}\right\vert
}+1\right) }{\log m_{A}}\left\Vert \frac{1}{l_{m_{A}-2^{\left\vert
m_{A}\right\vert }}}\sum\limits_{j=1}^{m_{A}-2^{\left\vert m_{A}\right\vert
}-1}\frac{D_{n-j}}{j}\right\Vert _{1}<\infty .  \label{conszeged}
\end{equation}%
Then subsequence of Nörlund logarithmic means $L_{m_{A}}\left( f\right) $
converges in the norm of space $X\left( \mathbb{I}\right) $.\newline
\newline
b) If the condition (\ref{conszeged}) does not holds then we can find a
function from the space $X\left( \mathbb{I}\right) $ for which the
convergence of logarithmic means $L_{m_{A}}\left( f\right) $ in the norm of
space $X\left( \mathbb{I}\right) $ does not holds.
\end{GT2}

Since 
\begin{eqnarray*}
&&\sup\limits_{A}\frac{\log \left( m_{A}-2^{\left\vert m_{A}\right\vert
}+1\right) }{\log m_{A}}\left\Vert \frac{1}{l_{m_{A}-2^{\left\vert
m_{A}\right\vert }}}\sum\limits_{j=1}^{m_{A}-2^{\left\vert m_{A}\right\vert
}-1}\frac{D_{n-j}}{j}\right\Vert _{1} \\
&\sim &\sup\limits_{A}\frac{\log \left( m_{A}-2^{\left\vert m_{A}\right\vert
}+1\right) }{\log m_{A}}\frac{1}{\left\vert m_{A}-2^{\left\vert
m_{A}\right\vert }\right\vert }\sum\limits_{k=1}^{\left\vert
m_{A}-2^{\left\vert m_{A}\right\vert }\right\vert }\left\vert \varepsilon
_{k}\left( n\right) -\varepsilon _{k+1}\left( n\right) \right\vert
l_{n\left( k-1\right) } \\
&\sim &\sup\limits_{A}\frac{1}{\left\vert m_{A}\right\vert }%
\sum\limits_{k=1}^{\left\vert m_{A}\right\vert }\left\vert \varepsilon
_{k}\left( n\right) -\varepsilon _{k+1}\left( n\right) \right\vert
l_{n\left( k-1\right) },
\end{eqnarray*}

from Theorem GT we can formulate necessary and sufficint condition for norm
convergence of subsequence of Nörlund logarithmic means.

\begin{theorem}
Let $f\in X\left( \mathbb{I}\right) $. Then the condition $%
\sup\limits_{A}V_{L}\left( m_{A}\right) <\infty $ is neccessary and
sufficient for convergence subsequence of Nörlund logarithmic means of
Walsh-Fourier series in norm of space $X\left( \mathbb{I}\right) .$
\end{theorem}


\end{document}